\documentclass[11pt]{amsart}
\usepackage{xcolor,mathrsfs}

\usepackage{tikz}
\usepackage{cleveref}

\newtheorem{Thm}{Theorem}
\newtheorem{Prop}[Thm]{Proposition}
\newtheorem{Lemma}[Thm]{Lemma}
\newtheorem{Cor}[Thm]{Corollary}

\theoremstyle{definition}
\newtheorem{Def}[Thm]{Definition}
\newtheorem{Conj}[Thm]{Conjecture}

\theoremstyle{remark}
\newtheorem{Example}{Example}
\newtheorem*{Remark}{Remark}

\title{Universal enveloping of a graded Lie algebra}
\author{Felipe Yukihide Yasumura}
\address{Department of Mathematics, Instituto de Matem\'atica e Estat\'istica, Universidade de S\~ao Paulo, SP, Brazil}
\email{fyyasumura@ime.usp.br}
\thanks{Supported by S\~ao Paulo Research Foundation (FAPESP), grant 2018/23690-6.}

\begin{document}
\begin{abstract}
In this paper we construct a graded universal enveloping algebra of a $G$-graded Lie algebra, where $G$ is not necessarily an abelian group. If the grading group is abelian, then it coincides with the classical construction. We prove the existence and uniqueness of the graded enveloping algebra. As consequences, we prove a graded variant of Witt's Theorem on the universal enveloping algebra of the free Lie algebra, and the graded version of Ado's Theorem, which states that every finite-dimensional Lie algebra admits a faithful finite dimensional representation. Furthermore we investigate if a Lie grading is equivalent to an abelian grading.
\end{abstract}
\maketitle

\section{Introduction}
The universal enveloping algebra of a Lie algebra is a classical and important construction, and it associates the representations of a Lie algebra to representations of an associative algebra, see any standard book on Lie algebra, for instance, \cite{Jac1979}. In this paper, we investigate the construction but starting with a $G$-graded Lie algebra where $G$ is a group. It is well-known that, if $G$ is abelian, then the universal enveloping algebra of a Lie algebra inherits a natural $G$-grading. However, it is known that the universal enveloping algebra need not be a graded algebra if the grading group $G$ is not abelian.

Let $G$ be an abelian group and $\mathcal{L}$ a $G$-graded Lie algebra. Then we can consider an ordered homogeneous basis $\{e_i\mid i\in I\}$ of $\mathcal{L}$. Let $[e_i,e_j]=\sum_{\ell}\alpha_{ij\ell}e_\ell$. Denote by $X_G=\bigcup_{g\in G}X_g$ where $X_g=\{x_1^{(g)},x_2^{(g)},\ldots\}$, and by $\mathbb{F}\langle X_G\rangle$ the free $G$-graded associative algebra over the field $\mathbb{F}$. The universal enveloping algebra of $\mathcal{L}$ is $U(\mathcal{L})\cong\mathbb{F}\langle X_G\rangle/J$, where $J$ is the (ordinary) ideal generated by all $\{x_ix_j-x_jx_i-\sum_{\ell\in I}\alpha_{ij\ell}x_\ell\mid i,j\in I\}$. Since the elements generating $J$ are homogeneous, we get that $J$ is a graded ideal. Thus, $U(\mathcal{L})$ is a $G$-graded algebra.

We extend the above construction for a $G$-graded Lie algebra where $G$ is a not necessarily abelian group (see \Cref{gr_universal} and \Cref{strong_gr_universal}). Our construction agrees with the classical one when the group is abelian. We find a PBW basis (\Cref{gradedPBW}). As a consequence, we prove an adapted version of Witt's Theorem in the context of the free $G$-graded Lie algebra (\Cref{completewitt}). We also deduce a graded version of Ado's Theorem, that is, we show that every finite-dimensional $G$-graded Lie algebra is a graded vector subspace of a finite-dimensional $G$-graded associative algebra (\Cref{graded_ado}). We study the problem if every group grading on a Lie algebra is equivalent to an abelian group grading. We apply our constructions and find an equivalent formulation to this problem (\Cref{problem_equiv}).

\section{Notations and Preliminaries}
Let $G$ be a group and $\mathcal{A}$ an algebra (not necessarily associative nor Lie), over a field $\mathbb{F}$. A \emph{$G$-grading} on $\mathcal{A}$ is a vector space decomposition $\mathcal{A}=\bigoplus_{g\in G}\mathcal{A}_g$ such that $\mathcal{A}_g\mathcal{A}_h\subseteq\mathcal{A}_{gh}$, for all $g$, $h\in G$. A $G$-graded algebra is an algebra endowed with a $G$-grading. The vector subspace $\mathcal{A}_g$ is called the \emph{homogeneous component} of degree $g$, and its nonzero elements are called homogeneous of degree $g$. When the decomposition is not explicitly given, we shall write $(\mathcal{A})_g$ to denote the homogeneous component of degree $g$. Given $0\ne x\in\mathcal{A}_g$, we denote $\deg_Gx=g$, or simply $\deg x=g$ when there is no risk of ambiguity. The \emph{support} of the grading is defined as
\[
\mathrm{Supp}\,\Gamma=\{g\in G\mid\mathcal{A}_g\ne0\}.
\]
A vector subspace $\mathcal{S}\subseteq\mathcal{A}$ is said to be \emph{graded} if $\mathcal{S}=\bigoplus\mathcal{A}_g\cap\mathcal{S}$. A \emph{graded subalgebra} is a subalgebra which is a graded subspace. Analogously we define a \emph{graded ideal}. If $\mathcal{S}$ is a graded ideal of $\mathcal{A}$, then the quotient $\mathcal{A}/\mathcal{S}$ inherits the structure of $G$-graded algebra.

Given two $G$-graded algebras $\mathcal{A}=\bigoplus_{g\in G}\mathcal{A}_g$ and $\mathcal{C}=\bigoplus_{g\in G}\mathcal{C}_g$, a $G$-graded homomorphism is an algebra homomorphism $\psi:\mathcal{A}\to\mathcal{C}$ such that $\psi(\mathcal{A}_g)\subseteq\mathcal{C}_g$, for all $g\in G$.

Let $\Gamma:\mathcal{A}=\bigoplus_{g\in G}\mathcal{A}_g$ and $\Gamma':\mathcal{A}=\bigoplus_{h\in H}\mathcal{A}_h'$ be two gradings on the same algebra $\mathcal{A}$. We say that $\Gamma$ and $\Gamma'$ are \emph{equivalent} if there exists an algebra automorphism $\psi$ of $\mathcal{A}$, such that for each $g\in G$ there exists $h\in H$ satisfying $\psi(\mathcal{A}_g)=\mathcal{A}_h'$. We say that a grading $\Gamma$ is \emph{realized} by a group $G$ if there exists a $G$-grading equivalent to $\Gamma$. We say that a grading is \emph{abelian} if it is realized by an abelian group. 

It is worth mentioning that if two gradings on the same algebra $\mathcal{A}$ are equivalent by an isomorphism $\psi$ and $I\subseteq\mathcal{A}$ is a graded ideal, then $I$ and $\psi(I)$ are equivalent and so are $\mathcal{A}/I$ and $\mathcal{A}/\psi(I)$.

\subsection{Universal group of a grading} We follow \cite{EK13} in this and next subsections. Given a $G$-grading, it is relevant to consider the group with minimal amount of relations that realizes the grading. Here is the formal  definition.
\begin{Def}[{\cite[Definition 1.17]{EK13}}]
Let $\Gamma$ be a $G$-grading on an algebra $\mathcal{A}$. The \emph{universal grading group} of $\Gamma$ is the group $U(\Gamma)$ such that, for every realization of $\Gamma$ as an $H$-grading, there exists a unique group homomorphism $U(\Gamma)\to H$ that is the identity map on $\mathrm{Supp}\,\Gamma$.
\end{Def}
The universal grading group of a group grading $\Gamma$ always exists. It may be taken as the group with the set of generators $\mathrm{Supp}\,\Gamma$, and relations $s_1s_2=s_3$, for $s_1$, $s_2$, $s_3\in\mathrm{Supp}\,\Gamma$ whenever $0\ne\mathcal{A}_{s_1}\mathcal{A}_{s_2}\subseteq\mathcal{A}_{s_3}$ (see \cite[Proposition 1.18]{EK13}).

It is also relevant to consider the \emph{universal abelian grading group} of a grading $\Gamma$. It is defined by the abelianization of $U(\Gamma)$, that is, $U_\mathrm{ab}(\Gamma)=U(\Gamma)/U(\Gamma)'$, where $G'$ is the commutator group of the group $G$. It is not always true that the universal abelian group grading realizes the grading $\Gamma$. Indeed, some homogeneous components may coalesce under this new group. The grading is realized by $U_\mathrm{ab}(\Gamma)$ if and only if the initial grading $\Gamma$ is abelian.

\subsection{Refinement and coarsening} Let $\Gamma:\mathcal{A}=\bigoplus_{g\in G}\mathcal{A}_g$ and $\Gamma':\mathcal{A}=\bigoplus_{h\in H}\mathcal{A}_h'$ be two gradings on $\mathcal{A}$. 
We say that $\Gamma'$ is a \emph{refinement} of $\Gamma$ (or that $\Gamma$ is a \emph{coarsening} of $\Gamma'$) if for every $h\in\mathrm{Supp}\,\Gamma'$, there exists $g\in\mathrm{Supp}\,\Gamma$ such that $\mathcal{A}_h'\subseteq\mathcal{A}_g$. The following is relevant to us:
\begin{Lemma}[{\cite[Part of Proposition 1.25]{EK13}}]\label{morphism_coarsening}
Let $\Gamma$ and $\Gamma'$ be gradings on $\mathcal{A}$, and assume that $\Gamma'$ is realized as a $H$-grading and is a coarsening of $\Gamma$. Then, there exists a group epimorphism $p:U(\Gamma)\to H$ such that $p(\mathrm{Supp}\,\Gamma)=\mathrm{Supp}\,\Gamma'$.
\end{Lemma}

\subsection{Free algebra} For each $g\in G$, we let $X_g=\{x_1^{(g)},x_2^{(g)},\ldots\}$, and $X_G=\bigcup_{g\in G}X_g$. Then, the free associative algebra $\mathbb{F}\langle X_G\rangle$, freely generated by $X_G$, has a natural $G$-grading where the monomials are homogeneous and
$$
\deg_G x_{i_1}^{(g_1)}\cdots x_{i_m}^{(g_m)}=g_1\cdots g_m.
$$
It satisfies the following universal property: for any associative $G$-graded algebra $\mathcal{A}$, and each map $\psi_0\colon X_G\to\mathcal{A}$ respecting the degrees (that is, $\psi_0(x_i^{(g)})\in\mathcal{A}_g$), there exists a unique $G$-graded algebra homomorphism $\psi\colon \mathbb{F}\langle X_G\rangle\to\mathcal{A}$ extending $\psi_0$. Hence, we can define a $G$-graded polynomial identity of a $G$-graded associative algebra as an element of the set
\[
\bigcap_{\substack{\psi\colon \mathbb{F}\langle X_G\rangle\to\mathcal{A}\\\text{$G$-graded homomorphism}}}\mathrm{Ker}\,\psi.
\]

\subsection{Free pair} We recall the notion of a  variety of associative-Lie pairs. An associative Lie-pair is a pair $(\mathcal{L},\mathcal{A})$ where $\mathcal{A}$ is an associative algebra generated by $\mathcal{L}$, and $\mathcal{L}$ is a Lie algebra which is a vector subspace of $\mathcal{A}$, and the product of the $\mathcal{L}$ coincides with the restriction of the commutator of $\mathcal{A}$ to $\mathcal{L}$.

Given two pairs $(\mathcal{L},\mathcal{A})$ and $(\mathcal{H},\mathcal{C})$, a homomorphism of pairs is an algebra homomorphism $\psi\colon \mathcal{A}\to\mathcal{C}$ such that $\psi(\mathcal{L})\subseteq\mathcal{H}$. Hence, $\psi$ restricts to a Lie homomorphism $\mathcal{L}\to\mathcal{H}$.

Let $\mathscr{W}$ be a class of associative-Lie pairs. A pair $(\mathcal{L},\mathcal{F})$ is free in the class $\mathscr{W}$, freely generated by $X$, if for any pair $(\mathcal{H},\mathcal{C})$ in $\mathscr{W}$, and for any map $\psi_0\colon X\to\mathcal{H}$, there exists a unique homomorphism of pairs from $(\mathcal{L},\mathcal{F})$ to $(\mathcal{H},\mathcal{B})$ extending $\psi_0$. It is clear that $(\mathcal{L}(X),\mathbb{F}\langle X\rangle)$, where $\mathcal{L}(X)$ is the Lie subalgebra of the free associative algebra generated by $X$ (that is, the free Lie algebra) is a free pair, freely generated by $X$, in the class of all associative-Lie pairs. Now, due to the existence of a free associative-Lie pair, we can speak of polynomial identities of pairs. By tradition these are called \emph{weak polynomial identities} of the pair $(\mathcal{L},\mathcal{A})$ (or \emph{identities of representations of Lie algebras}).

The above discussion can be extended to the context of $G$-graded associative-Lie pairs $(\mathcal{L},\mathcal{A})$: a $G$-graded associative algebra $\mathcal{A}$, and a $G$-graded Lie algebra $\mathcal{L}$ that is a $G$-graded subspace of $\mathcal{A}$, where the product of $\mathcal{L}$ is the commutator of $\mathcal{A}$ restricted to $\mathcal{L}$. We shall prove below that there exists a $G$-graded free associative-Lie pair.

\section{Graded Universal Enveloping Algebra}
We let $G$ be a non-necessarily abelian group.

\noindent\textbf{Notation.} By a \emph{$G$-pair} $(\mathcal{L},\mathcal{A})$, we understand a $G$-graded associative-Lie pair, i.e. $\mathcal{A}$ is a $G$-graded associative algebra, $\mathcal{L}$ is a $G$-graded subspace of $\mathcal{A}$ such that $\mathcal{L}$ is also a $G$-graded Lie algebra with respect to the commutator of $\mathcal{A}$ restricted to $\mathcal{L}$. Equivalently, a $G$-pair is a pair $(\mathcal{L},\iota)$, where $\mathcal{L}$ is a $G$-graded Lie algebra, $\mathcal{A}$ is a $G$-graded associative algebra, and $\iota\colon\mathcal{L}\to\mathcal{A}$ is a graded linear map such that $\iota\colon\mathcal{L}\to\mathcal{A}^{(-)}$ is an (ungraded) Lie monomorphism.

Note that the $G$-grading on the vector space $\mathcal{A}$ does not necessarily define a $G$-graded Lie algebra with respect to the commutator. If this is the case, then we shall call the $G$-pair $(\mathcal{L},\mathcal{A})$ a \emph{strong $G$-pair}.

\begin{Def}
Let $\mathcal{L}$ be a $G$-graded Lie algebra. The \emph{$G$-graded universal enveloping algebra} of $\mathcal{L}$ is a $G$-pair $(\mathcal{L}, U_G(\mathcal{L}))$ such that: (i) $U_G(\mathcal{L})$ is generated by $\mathcal{L}$, and (ii) for every $G$-pair $(\mathcal{H},\mathcal{A})$ and every $G$-graded Lie homomorphism $\mathcal{L}\to\mathcal{H}$, there exists an extension to a $G$-graded algebra homomorphism $U_G(\mathcal{L})\to\mathcal{A}$.
\end{Def}

As it is customary, we shall call the algebra $U_G(\mathcal{L})$ a $G$-graded universal enveloping algebra of $\mathcal{L}$. We shall prove the existence and uniqueness of $U_G(\mathcal{L})$. It will be an appropriate quotient of the usual universal enveloping algebra $U(\mathcal{L})$.

\begin{Remark}
In what follows, we may weaken the definition of a pair, replacing the condition of $\iota:\mathcal{L}\to\mathcal{A}^{(-)}$ being a monomorphism to ask $\iota$ to be only a homomorphism. This is because we cannot guarantee (until \Cref{gradedPBW}) that $\mathcal{L}\to U_G(\mathcal{L})$ is an embedding.
\end{Remark}

\begin{Lemma}\label{lem1}
Let $(\mathcal{L},\mathcal{A})$ be a $G$-pair and let $x$, $y\in\mathcal{L}$ be homogeneous elements such that $\deg x=g$, $\deg y = h$, and $gh\ne hg$. Then $xy=0$.
\end{Lemma}
\begin{proof}
Write $\mathcal{L}=\bigoplus_{g\in G}\mathcal{L}_g$ and $\mathcal{A}=\bigoplus_{g\in G}\mathcal{A}_g$, then $\mathcal{L}_g\subseteq\mathcal{A}_g$, for all $g\in G$.

By $g=\deg x$ and $h=\deg y$ we get $yx\in\mathcal{A}_{hg}$ and $xy\in\mathcal{A}_{gh}$. On the other hand,
\[
-[y,x]=-yx+xy \in \mathcal{L}_{hg}\subseteq \mathcal{A}_{hg}.
\]
Thus $xy=yx+(-yx+xy)\in\mathcal{A}_{hg}\cap\mathcal{A}_{gh}=0$.
\end{proof}
In the language of polynomial identities, we have:
\begin{Cor}\label{weakidentities}
Let $(\mathcal{L},\mathcal{A})$ be a $G$-pair. Then the pair satisfies the weak $G$-graded polynomial identities
$$
x_1^{(g)}x_2^{(h)}=0,\quad [g,h]\ne1.\qed
$$
\end{Cor}

Let $\mathcal{L}$ be a $G$-graded Lie algebra with a homogeneous vector space basis $\mathcal{B}=\{e_i\mid i\in N\}$, where $N$ is ordered, and denote $[e_i,e_j]=\sum_{\ell\in N}\alpha_{ij}^{(k)}e_k$. Let $X_\mathcal{B}=\{x_i\mid i\in I\}$. Let $\mathbb{F}\langle X_\mathcal{B}\rangle$ be the free associative $G$-graded algebra, freely generated by $X_\mathcal{B}$, where $\deg x_i=\deg_Ge_i$, for each $i\in N$. Define $I_\mathcal{B}$ to be the graded ideal generated by
\begin{equation}\label{elem_gen}
\{x_ix_j-x_jx_i-\sum_{k}\alpha_{ij}^{(k)}x_k\mid i,j\in N\},
\end{equation}
that is, $I_\mathcal{B}$ is the least graded ideal containing all the elements above.  Denote further by $J$ the ungraded ideal generated by the same elements, so that $U(\mathcal{L})\cong\mathbb{F}\langle X_\mathcal{B}\rangle/J$. The algebra $\mathbb{F}\langle X_\mathcal{B}\rangle/I_\mathcal{B}$ is graded, and $(\mathcal{L},\mathbb{F}\langle X_\mathcal{B}\rangle/I_\mathcal{B})$ is a $G$-pair.

\begin{Lemma}\label{lem2}
Let $J_\mathcal{B}$ be the ideal of $U(\mathcal{L})$ generated by
\[
\{e_ie_j\mid i,j\in N,\,[\deg e_i,\deg e_j]\ne1\}.
\]
Then
\[
\mathbb{F}\langle X_\mathcal{B}\rangle/I_\mathcal{B}\cong U(\mathcal{L})/J_\mathcal{B}.
\]
\end{Lemma}
\begin{proof}
Since $J\subseteq I_\mathcal{B}$, there exists a surjective algebra homomorphism $\pi:U(\mathcal{L})\to\mathbb{F}\langle X_\mathcal{B}\rangle/I_\mathcal{B}$. By Lemma \ref{lem1}, since $(\mathcal{L},\mathbb{F}\langle X_\mathcal{B}\rangle/I_\mathcal{B})$ is a $G$-pair, then $J_G\subseteq\mathrm{Ker}\,\pi$. So, we get that
\[
I_\mathcal{B}\supseteq J+K,
\]
where $K$ is the (possibly ungraded) ideal $\langle x_ix_j\mid i,j\in N,\,[\deg x_i,\deg x_j]\ne1\rangle$. Note that $K$ is actually graded, since its generators are homogeneous elements. Now, each of the elements \eqref{elem_gen} is either homogeneous (this is the case whenever $[\deg x_i,\deg x_j]=1$), or it belongs to $K$ (if  $[\deg x_i,\deg x_j]\ne1$). Hence, $J+K$ is a graded ideal. Thus we obtain that $J+K\supseteq I_\mathcal{B}$, proving the result.
\end{proof}

Now we can prove the existence of a $G$-graded universal enveloping algebra of $\mathcal{L}$.

\begin{Lemma}
The $G$-pair ($\mathcal{L},\mathbb{F}\langle X_\mathcal{B}\rangle/I_\mathcal{B})$ is a $G$-graded universal enveloping algebra of $\mathcal{L}$.
\end{Lemma}
\begin{proof}
Let $(\mathcal{H},\mathcal{A})$ be a $G$-pair and $\varphi\colon\mathcal{L}\to\mathcal{H}$ be a $G$-graded Lie homomorphism. Then $\varphi$ admits an extension (also denoted by $\varphi$) to an algebra homomorphism $\varphi\colon U(\mathcal{L})\to\mathrm{alg}(\mathcal{H})\subseteq\mathcal{A}$.  Here $\mathrm{alg}(\mathcal{H})$ is the associative subalgebra generated by $\mathcal{H}$. If $[\deg e_i,\deg e_j]\ne1$, then Lemma \ref{lem1} tells us that
\[
0=\varphi(e_i)\varphi(e_j)=\varphi(e_ie_j).
\]
Thus, $\varphi$ factors through $J_\mathcal{B}$, that is, there exists a map $\bar{\varphi}\colon U(\mathcal{L})/J_\mathcal{B}\to\mathcal{A}$. By Lemma \ref{lem2}, $U(\mathcal{L})/J_\mathcal{B}\cong\mathbb{F}\langle X_\mathcal{B}\rangle/I_\mathcal{B}$. The map $\bar{\varphi}$ is a graded map and extends $\varphi$, concluding the proof.
\end{proof}

Finally, it is a standard exercise to prove that the $G$-graded universal enveloping algebra is unique, up to an isomorphism.
\begin{Lemma}
Let $U_G$ and $V_G$ be $G$-graded universal enveloping algebras of $\mathcal{L}$. Then there exists a $G$-graded isomorphism $\iota\colon U_G\to V_G$ such that $\iota(x)=x$, for each $x\in\mathcal{L}$.
\end{Lemma}
\begin{proof}
Since $(\mathcal{L},V_G)$ is a pair, the identity map $\mathcal{L}\to\mathcal{L}$ admits an extension to a $G$-graded algebra homomorphism $\iota\colon U_G\to V_G$. Conversely, the identity map also admits an extension to a $G$-graded homomorphism $j\colon V_G\to U_G$. Since the composition $j\iota\colon U_G\to U_G$ fixes all the elements of $\mathcal{L}$, it must be the identity map. Similarly, $\iota j$ is the identity map.
\end{proof}

We summarize our results. Of course, from a vector space basis of $U(\mathcal{L})$, we obtain a set of generators of $U_G(\mathcal{L})$ as a vector space.

\begin{Thm}\label{gr_universal}
Let $\mathcal{L}$ be a $G$-graded Lie algebra, where $G$ is a non-necessarily abelian group. Then it admits a unique, up to an isomorphism, $G$-graded universal enveloping algebra $U_G(\mathcal{L})$. If $\{e_i\mid i\in I\}$ is an ordered homogeneous basis of $\mathcal{L}$, then
\[
e_{i_1}\cdots e_{i_m}, m\ge0,\, i_1\le\cdots\le i_m,\,[\deg e_{i_j},\deg e_{i_{j+1}}]=1, j\in\{1,2,\ldots,m-1\},
\]
is a set of homogeneous elements that spans $U_G(\mathcal{L})$.\qed
\end{Thm}

\begin{Example}
If $G$ is abelian and $\mathcal{L}$ is a $G$-graded algebra, then $U_G(\mathcal{L})=U(\mathcal{L})$.
\end{Example}

\begin{Example}\label{abelian_lie_example}
Let $G=C_2\ast C_2=\langle g,h\mid g^2=h^2=1\rangle$, and $\mathcal{L}=\mathrm{Span}\{x,y\}$ be the $2$-dimensional abelian Lie algebra. Define a $G$-grading on $\mathcal{L}$ where $\deg x=g$ and $\deg y=h$. Then, it is known that $U(\mathcal{L})\cong\mathbb{F}[X,Y]$, the polynomial algebra in two commuting variables. Since $U_G(\mathcal{L})\cong\mathbb{F}[X,Y]/\langle XY\rangle$, one obtains that $U_G(\mathcal{L})$ is the associative unital subalgebra of $\mathbb{F}[X]\oplus\mathbb{F}[Y]$ generated by $X$ and $Y$, that is,
\[
U_G(\mathcal{L})=\{\alpha(1,1)+(F(X),G(Y))\mid\alpha\in\mathbb{F},F(X)\in\mathbb{F}[X],G(Y)\in\mathbb{F}[Y]\}.
\]
\end{Example}

\section{Strong pair\label{variety}}
We assume that $G$ is a not necessarily abelian group. We investigate when $U_G(\mathcal{L})$ is a $G$-graded Lie algebra with respect to the commutator. To this end, we need to investigate strong pairs, that is,  $G$-pairs $(\mathcal{L},\mathcal{A})$, where $\mathcal{A}$ is a $G$-graded Lie algebra with respect to the commutator. Let $\mathscr{V}_G$ denote the variety of $G$-graded associative algebras satisfying the polynomial identities
\[
\left\{x^{(g)}_1x_2^{(h)}=0\mid g,h\in G,[g,h]\ne1\right\}.
\]
Whenever we have a strong pair $(\mathcal{L},\mathcal{A})$, from \Cref{lem1}, we have $\mathcal{A}\in\mathscr{V}_G$.

Denote by $X_G=\bigcup_{g\in G}X_g$, where $X_g=\{x_1^{(g)},x_2^{(g)},\ldots\}$, for each $g\in G$. Let $\mathbb{F}_G(X_G)$ be the relatively free algebra in $\mathscr{V}_G$, freely generated by $X_G$. Then
$$
\mathbb{F}_G(X_G)\cong\mathbb{F}\langle X_G\rangle/\langle x_1^{(g)}x_2^{(h)}\mid [g,h]\ne1\rangle^{T_G}.
$$

We slightly modify the definition of $G$-graded universal enveloping algebra of a Lie algebra.
\begin{Def}
Let $\mathcal{L}$ be a $G$-graded Lie algebra. The \emph{strong $G$-graded universal enveloping} of $\mathcal{L}$ is a strong pair $(\mathcal{L},\mathrm{SU}_G(\mathcal{L}))$ such that for every strong pair $(\mathcal{H},\mathcal{A})$, each $G$-graded Lie homomorphism $\mathcal{L}\to\mathcal{H}$ admits a unique extension to a $G$-graded algebra homomorphism $\mathrm{SU}_G(\mathcal{L})\to\mathcal{A}$.
\end{Def}
A standard argument shows that if a strong $G$-graded universal enveloping algebra exists, then it is unique up to an isomorphism that is the identity map on $\mathcal{L}$. Moreover, $\mathcal{L}$ generates $\mathrm{SU}_G(\mathcal{L})$. If it exists, it is a quotient of the $G$-graded universal enveloping algebra.

\begin{Lemma}\label{quotient}
If a strong $G$-graded universal enveloping algebra of $\mathcal{L}$ exists, then it is a quotient of $U_G(\mathcal{L})$.
\end{Lemma}
\begin{proof}
If $(\mathcal{L},\mathrm{SU}_G(\mathcal{L}))$ is a strong pair, then it is also a $G$-pair. Hence there exists an extension of the identity map $\mathcal{L}\to\mathcal{L}$ to an algebra homomorphism $U_G(\mathcal{L})\to\mathrm{SU}_G(\mathcal{L})$. The statement follows since $\mathcal{L}$ generates $\mathrm{SU}_G(\mathcal{L})$.
\end{proof}

\begin{Example}
If $G$ is an abelian group, then every $G$-pair is a strong pair. Hence, $\mathrm{SU}_G(\mathcal{L})$ exists and it coincides with $U_G(\mathcal{L})$ (which in turn equals $U(\mathcal{L})$).
\end{Example}

\begin{Example}
Consider once  again \Cref{abelian_lie_example}: let $G=C_2\ast C_2=\langle g,h\mid g^2=h^2=1\rangle$ and $\mathcal{L}=\mathrm{Span}\{x,y\}$ the $2$-dimensional abelian Lie algebra, where $\deg x=g$ and $\deg y=h$. The $G$-graded universal enveloping algebra of $\mathcal{L}$ is a Lie algebra with respect to the commutator (indeed, $U_G(\mathcal{L})$ is a commutative algebra). In particular, it satisfies the definition of the strong $G$-graded universal enveloping algebra, so $\mathrm{SU}_G(\mathcal{L})=U_G(\mathcal{L})$.
\end{Example}

\begin{Example}
More generally, if $U_G(\mathcal{L})$ happens to be a $G$-graded Lie algebra with respect to the commutator, then $\mathrm{SU}_G(\mathcal{L})$ exists and it coincides with $U_G(\mathcal{L})$. Conversely, if $\mathrm{SU}_G(\mathcal{L})=U_G(\mathcal{L})$, then $U_G(\mathcal{L})^{(-)}$ is a $G$-graded Lie algebra.
\end{Example}

Now, we shall prove the existence of a strong $G$-graded universal enveloping algebra. The construction is similar to that of the $G$-graded universal enveloping algebra, but we need to work in the variety $\mathscr{V}_G$.

Let $\mathcal{L}$ be a $G$-graded Lie algebra, and let $\mathcal{B}=\{e_i\mid i\in N\}$ be a homogeneous basis of $\mathcal{L}$. Consider the structure constants $\alpha_{ij}^{(k)}$, $i$, $j$, $k\in N$, that is,
\[
[e_i,e_j]=\sum_{k\in N}\alpha_{ij}^{(k)}e_k.
\]
We let $X_\mathcal{B}=\{z_i=x_i^{(g_i)}\mid i\in N, \,g_i=\deg e_i\}$, and let $\mathbb{F}_G(X_\mathcal{B})$ be the relatively free algebra in $\mathscr{V}_G$, freely generated by $X_\mathcal{B}$. Let $J_\mathcal{B}$ be the ideal of $\mathbb{F}_G(X_\mathcal{B})$ generated by all
\[
z_iz_j-z_jz_i-\sum_{k\in N}\alpha_{ij}^{(k)}z_k,\quad i,j\in N.
\]
The above elements are homogeneous. Indeed, either $[\deg z_i,\deg z_j]=1$, or $z_iz_j=z_jz_i=[z_i,z_j]=0$. Thus $J_\mathcal{B}$ is a graded ideal.

\begin{Thm}\label{strong_gr_universal}
$\mathrm{SU}_G(\mathcal{L})=\mathbb{F}_G(X_\mathcal{B})/J_\mathcal{B}$.
\end{Thm}
\begin{proof}
Let $(\mathcal{H},\mathcal{A})$ be a strong pair, and let $f_0\colon \mathcal{L}\to\mathcal{H}$ be a $G$-graded Lie homomorphism. Then $f_0$ restricts to a map $X_\mathcal{B}\to\mathcal{H}\subseteq\mathcal{A}$ which  respects the degrees. Since $\mathcal{A}\in\mathscr{V}_G$, the restriction of $f_0$ extends to a $G$-graded algebra homomorphism $f\colon \mathbb{F}_G(X_\mathcal{B})\to\mathcal{A}$. Since both $\mathbb{F}_G(X_\mathcal{B})$ and $\mathcal{A}$ are $G$-graded Lie algebras with respect to the commutators, then $f$ factors through $J_\mathcal{B}$. Thus $f_0$ admits an extension to a $G$-graded algebra homomorphism $\mathbb{F}_G(X_\mathcal{B})/J_\mathcal{B}\to\mathcal{A}$.
\end{proof}
As a consequence, every $G$-graded Lie algebra is a subalgebra of some $\mathcal{A}^{(-)}$, where $\mathcal{A}$ is a $G$-graded associative algebra such that $\mathcal{A}^{(-)}$ is a $G$-graded Lie algebra.

\noindent\textbf{Notation.} Whenever it is convenient, we shall identify the homogeneous variables $z_i\in X_\mathcal{B}$ with the basis elements $e_i\in\mathcal{B}$. So, $\mathrm{SU}_G(\mathcal{L})$ is spanned by all elements $e_{i_1}\cdots e_{i_m}$, where $e_{i_1}$, \dots, $e_{i_m}\in\mathcal{B}$.

\newcommand{\gab}{g.a.s.~}
\section{A PBW basis for $\mathrm{SU}_G(\mathcal{L})$}

In this section, we shall find a homogeneous vector space basis for $\mathrm{SU}_G(\mathcal{L})$. For, we let $X_G=\bigcup_{g\in G}X_g$ be a set of $G$-homogeneous variables (where each $X_g$ is finite or not). We start finding a homogeneous vector space basis for $\mathbb{F}_G(X_G)$.

\noindent\textbf{Notation.} If $\{g_1,\ldots,g_m\}$ is a subset of a group, then the abbreviation \gab means that $\{g_1,\ldots,g_m\}$ generates an abelian subgroup.

\begin{Lemma}\label{generators}
The relatively free algebra $\mathbb{F}_G(X_G)$ has, as a $G$-homogeneous vector space basis, all monomials of the kind
$$
x_{i_1}^{(g_1)}\cdots x_{i_m}^{(g_m)},\quad\{g_1,\ldots,g_m\}\text{ \gab}.
$$
\end{Lemma}
\begin{proof}
It is clear that such a set generates $\mathbb{F}_G(X_G)$ as a vector space. So, it is sufficient to prove that it is linearly independent. Consider a finite subset $S$ of the above monomials. We may assume that a fixed set of variables, say $Y=\{x_{i_1}^{(g_1)},\ldots,x_{i_m}^{(g_m)}\}$, appears in all the monomials of $S$. We can assume that the subgroup generated by $\{g_1,\ldots,g_m\}$ is abelian, otherwise every monomial in $S$ would be zero. So, let $H=\langle g_1,\ldots,g_m\rangle$ be the abelian subgroup of $G$ generated by $g_1$, \dots, $g_m$. The free associative $H$-graded algebra $\mathbb{F}\langle Y\rangle$, freely generated by $Y$, may be seen as a $G$-graded algebra, where the homogeneous components in $G\setminus H$ are $0$. Therefore, there exists a surjective $G$-graded algebra homomorphism $p:\mathbb{F}_G(X_G)\to\mathbb{F}\langle Y\rangle$ via
$$
p(x_j^{(g)})=\left\{\begin{array}{ll}%
x_j^{(g)},&\text{ if $x_j^{(g)}\in Y$},\\%
0,&\text{ otherwise}.\end{array}\right.
$$
Now, the set $S$ is linearly independent under the image of $p$. So, $S$ is linearly independent as well. The proof is complete.
\end{proof}

Now, let $G$ be any group and $\mathcal{L}$ a $G$-graded Lie algebra over an arbitrary field $\mathbb{F}$, and let $\mathcal{B}=\{e_i\mid i\in N\}$ be a homogeneous vector space basis of $\mathcal{L}$, where $N$ is ordered. Let
$$
[e_i,e_j]=\sum_{\ell\in N}\alpha_{ij}^{(\ell)}e_\ell
$$
be the structure constants of $\mathcal{L}$ with respect to $\mathcal{B}$. We shall consider the set of free variables $X_\mathcal{B}=\{e_i\mid i\in N\}$, and use the same letters to denote the homogeneous variables and the elements of $\mathcal{B}$. Let $\mathbb{F}_G(X_\mathcal{B})$ be the relatively free $G$-graded algebra in $\mathscr{V}_G$, freely generated by $X_\mathcal{B}$, and $\mathbb{F}\langle X_\mathcal{B}\rangle$ be the free associative $G$-graded algebra, freely generated by $X_\mathcal{B}$.

We let $\mathfrak{R}$ be the vector subspace of $\mathbb{F}\langle X_\mathcal{B}\rangle$ spanned by all
$$
e_{i_1}\cdots e_{i_m},\quad m\ge0, \, i_1\le\cdots\le i_m.
$$
The following is a classical result.
\begin{Lemma}[{\cite[Lemma 5.2]{Jac1979}}]\label{lemma}
There exists a linear map $\sigma_0:\mathbb{F}\langle X_\mathcal{B}\rangle\to\mathfrak{R}$ such that
\begin{eqnarray*}
&\sigma_01=1,&\\%
&\sigma_0(e_{i_1}\cdots e_{i_m})=e_{i_1}\cdots e_{i_m},\text{ if }i_1\le\cdots\le i_m,&\\%
&\displaystyle\sigma_0(e_{j_1}\cdots e_{j_m}-e_{j_1}\cdots e_{j_{k+1}}e_{j_k}\cdots e_{j_m})=\sigma_0(e_{j_1}\cdots\left(\sum_{\ell\in N}\alpha_{j_kj_{k+1}}^{(\ell)}e_\ell\right)\cdots e_{j_m}).&
\end{eqnarray*}
\end{Lemma}

As an immediate consequence, we have the following graded version. Let $\mathfrak{B}$ be the vector subspace of $\mathbb{F}_G(X_\mathcal{B})$ spanned by all
$$
e_{i_1}\cdots e_{i_m},\quad m\ge0, \, i_1\le\cdots\le i_m,\,\{\deg_Ge_{i_1},\ldots,\deg_Ge_{i_m}\}\text{ \gab}
$$
\begin{Lemma}\label{gradedlemma}
There exists a linear map $\sigma:\mathbb{F}_G(X_\mathcal{B})\to\mathfrak{B}$ such that
$$
\sigma(e_{i_1}\cdots e_{i_m})=0,
$$
if $\{\deg_Ge_{i_1},\ldots,\deg_Ge_{i_m}\}$ does not generate an abelian subgroup, and otherwise,
\begin{eqnarray*}
&\sigma1=1,&\\%
&\sigma(e_{i_1}\cdots e_{i_m})=e_{i_1}\cdots e_{i_m},\text{ if }i_1\le\cdots\le i_m,&\\%
&\displaystyle\sigma(e_{j_1}\cdots e_{j_m}-e_{j_1}\cdots e_{j_{k+1}}e_{j_k}\cdots e_{j_m})=\sigma(e_{j_1}\cdots\left(\sum_{\ell\in N}\alpha_{j_kj_{k+1}}^{(\ell)}e_\ell\right)\cdots e_{j_m}).&
\end{eqnarray*}
\end{Lemma}
\begin{proof}
The projection $\mathbb{F}\langle X_\mathcal{B}\rangle\to\mathbb{F}_G(X_\mathcal{B})$ induces a map $\mathfrak{R}\to\mathfrak{B}$. Let $\sigma_0:\mathbb{F}\langle X_\mathcal{B}\rangle\to\mathfrak{R}$ be the linear map of \Cref{lemma}. Then, we have the diagram

\tikzset{every picture/.style={line width=0.75pt}} 

\begin{tikzpicture}[x=0.75pt,y=0.75pt,yscale=-1,xscale=1]
\path (0,110);

\draw (721,21) node    {$0$};
\draw (701,71) node    {$0$};
\draw (202,59) node    {$\mathbb{F}\langle X_{\mathcal{B}} \rangle $};
\draw (322,60) node    {$\mathfrak{R}$};
\draw (315,129) node [anchor=north west][inner sep=0.75pt]   [align=left] {$\displaystyle \mathfrak{B}$};
\draw (189,129) node [anchor=north west][inner sep=0.75pt]   [align=left] {$\displaystyle \mathbb{F}_G(X_\mathcal{B})$};
\draw    (231.5,59.25) -- (310,59.9) ;
\draw [shift={(312,59.92)}, rotate = 180.48] [color={rgb, 255:red, 0; green, 0; blue, 0 }  ][line width=0.75]    (10.93,-3.29) .. controls (6.95,-1.4) and (3.31,-0.3) .. (0,0) .. controls (3.31,0.3) and (6.95,1.4) .. (10.93,3.29)   ;
\draw    (322,72) -- (322,123) ;
\draw [shift={(322,125)}, rotate = 270] [color={rgb, 255:red, 0; green, 0; blue, 0 }  ][line width=0.75]    (10.93,-3.29) .. controls (6.95,-1.4) and (3.31,-0.3) .. (0,0) .. controls (3.31,0.3) and (6.95,1.4) .. (10.93,3.29)   ;
\draw    (223.27,73) -- (310.33,130.32) ;
\draw [shift={(312,131.42)}, rotate = 213.36] [color={rgb, 255:red, 0; green, 0; blue, 0 }  ][line width=0.75]    (10.93,-3.29) .. controls (6.95,-1.4) and (3.31,-0.3) .. (0,0) .. controls (3.31,0.3) and (6.95,1.4) .. (10.93,3.29)   ;
\draw    (202.44,73) -- (204.03,123) ;
\draw [shift={(204.09,125)}, rotate = 268.19] [color={rgb, 255:red, 0; green, 0; blue, 0 }  ][line width=0.75]    (10.93,-3.29) .. controls (6.95,-1.4) and (3.31,-0.3) .. (0,0) .. controls (3.31,0.3) and (6.95,1.4) .. (10.93,3.29)   ;

\end{tikzpicture}

Since the kernel of $\mathbb{F}\langle X_\mathcal{B}\rangle\to\mathfrak{B}$ lies inside the kernel of $\mathbb{F}\langle X_\mathcal{B}\rangle\to\mathbb{F}_G(X_\mathcal{B})$ (see \Cref{generators}), we obtain the required map $\sigma:\mathbb{F}_G(X_\mathcal{B})\to\mathfrak{B}$.
\end{proof}

This gives a PBW basis for $\mathrm{SU}_G(\mathcal{L})$. More precisely, we have
\begin{Thm}\label{gradedPBW}
Let $G$ be any group and $\mathcal{L}$ be a $G$-graded Lie algebra over a field $\mathbb{F}$. Let $\{e_i\mid i\in N\}$ be a homogeneous basis of $\mathcal{L}$, where $N$ is ordered. Then, a homogeneous vector space basis of $\mathrm{SU}_G(\mathcal{L})$ is given by
$$
e_{i_1}\cdots e_{i_m},\quad i_1\le\cdots\le i_m,\,\{\deg_Ge_{i_1},\ldots,\deg_Ge_{i_m}\}\text{ \gab}
$$
\end{Thm}
\begin{proof}
It is clear, as in the classical case, that the above monomials span $\mathrm{SU}_G(\mathcal{L})$. Now, the map $\sigma:\mathbb{F}_G(X_\mathcal{B})\to\mathfrak{B}$ from \Cref{gradedlemma} factors through $\mathrm{SU}_G(\mathcal{L})\to\mathfrak{B}$. So, we obtain a surjective linear map $\mathrm{SU}_G(\mathcal{L})\to\mathfrak{B}$, where the above set of monomials is sent to a linearly independent set. Hence, it is a basis.
\end{proof}

As a consequence, we obtain that $(\mathcal{L},U_G(\mathcal{L}))$ and $(\mathcal{L},\mathrm{SU}_G(\mathcal{L}))$ are indeed pairs. That is, we have the following.
\begin{Cor}
Let $G$ be any group and $\mathcal{L}$ a $G$-graded Lie algebra. Then
\begin{enumerate}
\renewcommand{\labelenumi}{(\roman{enumi})}
\item there is an embedding $\mathcal{L}\hookrightarrow\mathrm{SU}_G(\mathcal{L})$,
\item there is an embedding $\mathcal{L}\hookrightarrow U_G(\mathcal{L})$.
\end{enumerate}
\end{Cor}
\begin{proof}
The assertion (i) follows directly from the previous theorem. To prove (ii), one should consider the diagram

\tikzset{every picture/.style={line width=0.75pt}} 

\begin{tikzpicture}[x=0.75pt,y=0.75pt,yscale=-1,xscale=1]
\path (0,45); 

\draw (721,21) node    {$0$};
\draw (701,71) node    {$0$};
\draw (170,46) node    {$\mathcal{L}$};
\draw (290,47) node    {$U_{G}(\mathcal{L})$};
\draw (261,120) node [anchor=north west][inner sep=0.75pt]   [align=left] {$\displaystyle SU_{G}(\mathcal{L})$};
\draw    (180,46.08) -- (260.5,46.75) ;
\draw [shift={(262.5,46.77)}, rotate = 180.48] [color={rgb, 255:red, 0; green, 0; blue, 0 }  ][line width=0.75]    (10.93,-3.29) .. controls (6.95,-1.4) and (3.31,-0.3) .. (0,0) .. controls (3.31,0.3) and (6.95,1.4) .. (10.93,3.29)   ;
\draw    (290,60) -- (290,114) ;
\draw [shift={(290,116)}, rotate = 270] [color={rgb, 255:red, 0; green, 0; blue, 0 }  ][line width=0.75]    (10.93,-3.29) .. controls (6.95,-1.4) and (3.31,-0.3) .. (0,0) .. controls (3.31,0.3) and (6.95,1.4) .. (10.93,3.29)   ;
\draw    (180,52.92) -- (269.56,114.86) ;
\draw [shift={(271.2,116)}, rotate = 214.67] [color={rgb, 255:red, 0; green, 0; blue, 0 }  ][line width=0.75]    (10.93,-3.29) .. controls (6.95,-1.4) and (3.31,-0.3) .. (0,0) .. controls (3.31,0.3) and (6.95,1.4) .. (10.93,3.29)   ;

\end{tikzpicture}

where the surjective map $U_G(\mathcal{L})\to\mathrm{SU}_G(\mathcal{L})$ is from \Cref{quotient}.
\end{proof}

\section{Free graded Lie algebra}
There is a direct way to construct the $G$-graded free Lie algebra. It was mentioned, for instance, in \cite{Gord}. We include its construction here for the sake of completeness. Let $\mathbb{F}\{X_G\}$ be the absolutely free (nonassociative) $G$-graded algebra, and let $I$ be the graded T-ideal generated by all elements of the following types:
\begin{align*}
&x_1^{(g_1)}x_2^{(g_2)}+x_2^{(g_2)}x_1^{(g_1)},\\
&(x_1^{(g_1)}x_2^{(g_2)})x_3^{(g_3)}+(x_2^{(g_2)}x_3^{(g_3)})x_1^{(g_1)}+(x_3^{(g_3)}x_1^{(g_1)})x_2^{(g_2)}.
\end{align*}
Denote $\mathcal{L}(X_G)=\mathbb{F}\{X_G\}/I$. By construction, $\mathcal{L}(X_G)$ is a $G$-graded algebra, and it is clearly a Lie algebra.
\begin{Prop}
$\mathcal{L}(X_G)$ is free in the class of all $G$-graded Lie algebras, and it is freely generated by the set $X_G$.
\end{Prop}
\begin{proof}
Let $\mathcal{H}$ be a $G$-graded Lie algebra, and let $\varphi\colon X_G\to\mathcal{H}$ be a degree-preserving map, that is, $\deg\varphi(x_i^{(g)})=g$. Then $\varphi$ admits an extension to a $G$-graded homomorphism $\varphi\colon \mathbb{F}\{X_G\}\to\mathcal{H}$. Then $\mathrm{Ker}\,\varphi$ is a graded ideal; and, since $\mathcal{H}$ is a graded Lie algebra, it contains the generators of $I$. Thus, $\varphi$ factors through $\mathbb{F}\{X_G\}/I$, that is, $\varphi$ induces a graded Lie homomorphism $\mathcal{L}(X_G)\to\mathcal{H}$ that extends the map $X_G\to\mathcal{H}$.
\end{proof}

\subsection{Free graded algebras and pairs} 
We fix a group $G$. As a consequence of \Cref{weakidentities}, we may consider the variety $\mathscr{W}_G$ of pairs satisfying the weak $G$-graded polynomial identities
\[
x_1^{(g)}x_2^{(h)}=0,\quad [g,h]\ne1.
\]
Let $\mathcal{L}(X_G)$ be the free $G$-graded Lie algebra, and let  $U_G(\mathcal{L}(X_G))$ be its respective $G$-graded universal enveloping algebra.
\begin{Prop}\label{prop1}
The pair $(\mathcal{L}(X_G),U_G(\mathcal{L}(X_G)))$ is free in the variety of pairs $\mathscr{W}_G$, and $X_G$ is a set of free generators.
\end{Prop}
\begin{proof}
Let $(\mathcal{H},\mathcal{A})$ be a $G$-pair, and let $\psi_0\colon X_G\to \mathcal{H}$ be a degree-preserving map. Since $\mathcal{L}(X_G)$ is free, then $\psi_0$ admits a unique extension to a Lie homomorphism $\mathcal{L}(X_G)\to\mathcal{H}$. So, it admits a unique extension to an algebra homomorphism $U_G(\mathcal{L}(X_G))\to\mathcal{A}$. Thus, $(\mathcal{L}(X_G),U_G(\mathcal{L}(X_G)))$ is free in the variety of pairs $\mathscr{W}_G$, freely generated by $X_G$.
\end{proof}

Conversely, we may obtain the free $G$-graded Lie algebra from a free pair in $\mathscr{W}_G$. So, let $(\mathcal{L},\mathcal{F})$ be a free pair in the variety of $G$-graded pairs in $\mathscr{W}_G$, freely generated by $X$. We let $\mathcal{L}(X)$ be the $G$-graded Lie algebra generated by $X$. Then, $\mathcal{L}(X)=\mathcal{L}$. More precisely, we have:
\begin{Prop}\label{prop2}
The algebra $\mathcal{L}(X)$ is the free $G$-graded Lie algebra, freely generated by $X$, and the $G$-pair $(\mathcal{L}(X),U_G(\mathcal{L}(X)))$ is free in the variety of pairs $\mathscr{W}_G$, and $X$ is a set of free generators. 
\end{Prop}
\begin{proof}
Let $(\mathcal{H},\mathcal{A})$ be a $G$-pair, and let $\psi_0\colon X\to \mathcal{H}$ be a degree-preserving map. Since the pair belongs to $\mathscr{W}_G$ (\Cref{weakidentities}) and $(\mathcal{L},\mathcal{F})$ is free, there exists an extension of $\psi_0$ to a $G$-graded algebra homomorphism $\psi\colon \mathcal{F}\to\mathcal{A}$, which restricts to a Lie homomorphism $\bar{\psi}\colon \mathcal{L}\to\mathcal{H}$. In particular, $\psi$ restricts to a Lie homomorphism from $\mathcal{L}(X)\to\mathcal{H}$. Every $G$-graded Lie algebra may be thought of as the pair given by itself and its $G$-graded universal enveloping algebra, we get that $\mathcal{L}(X)$ is a free $G$-graded Lie algebra. Now $\bar{\psi}$ extends uniquely to a homomorphism $U_G(\mathcal{L}(X))\to\mathcal{A}$. Hence $(\mathcal{L}(X),U_G(\mathcal{L}(X)))$ is a free pair in $\mathscr{W}_G$, and it is freely generated by $X$.
\end{proof}

As a consequence of \Cref{prop1} and \Cref{prop2}, we obtain an alternative description of the free $G$-graded Lie algebra, and a relation with the free $G$-pair in the variety $\mathscr{W}_G$. This is a graded version of Witt's Theorem (see, for instance, \cite[Theorem 1.3.5]{Drenskybook} or \cite[Theorem V.7]{Jac1979} for the classical case).
\begin{Cor}\label{weakwitt}
Let $(\mathcal{L},\mathcal{F})$ be free in the variety of pairs $\mathscr{W}_G$, freely generated by $X$. Then $\mathcal{L}$ coincides with the Lie algebra $\mathcal{L}(X)$, generated by $X$, and it is the free $G$-graded Lie algebra, freely generated by $X$, and $\mathcal{F}=U_G(\mathcal{L})$.\qed
\end{Cor}

\subsection{Strong Witt's Theorem} Now, we obtain an outright graded version of Witt's Theorem. Let $\mathbb{F}_G(X_G)$ be the relatively free $G$-graded associative algebra in the variety $\mathscr{V}_G$, and let $\mathcal{L}(X_G)$ be its Lie subalgebra generated by $X_G$. Then we have
\begin{Thm}\label{strongwitt}
The algebra $\mathcal{L}(X_G)$ is a free $G$-graded Lie algebra, freely generated by $X_G$. Moreover, $\mathrm{SU}_G(\mathcal{L}(X_G))=\mathbb{F}_G(X_G)$.
\end{Thm}
\begin{proof}
We let $\mathcal{H}$ be a $G$-graded Lie algebra and $f_0:X_G\to\mathcal{H}$ a map respecting degrees. Now, since $\mathcal{H}\subseteq\mathrm{SU}_G(\mathcal{L})$, the map $f_0$ extends to a $G$-graded algebra homomorphism $f:\mathbb{F}_G(X_G)\to\mathrm{SU}_G(\mathcal{H})$. Such maps restricts to a $G$-graded Lie homomorphism $\mathcal{L}(X_G)\to\mathrm{Lie}(\mathcal{H})=\mathcal{H}$. So, $\mathcal{L}(X_G)$ is free, freely generated by $X_G$.

Now, let $(\mathcal{H},\mathcal{A})$ be a strong pair, and $f:\mathcal{L}(X_G)\to\mathcal{H}$ a $G$-graded Lie homomorphism. Then, $f$ restricts to a graded map $X_G\to\mathcal{H}$, which admits an extension to a $G$-graded algebra homomorphism $\bar{f}:\mathbb{F}_G(X_G)\to\mathcal{A}$. Since the restriction of $\bar{f}$ to $\mathcal{L}(X_G)$ is a $G$-graded Lie homomorphism, and $f$ is the unique extension of the map $X_G\to\mathcal{H}$, we get that the restriction of $\bar{f}$ coincides with $f$. It means that $\bar{f}$ is an extension of $f$. So, $\mathrm{SU}_G(\mathcal{L}(X_G))=\mathbb{F}_G(X_G)$.
\end{proof}

Combining \Cref{weakwitt} and \Cref{strongwitt} we obtain a complete version of the graded version of Witt's Theorem. We summarize the results, recalling the main definitions.
\begin{Cor}\label{completewitt}
Let $G$ be a group, and consider the set of $G$-graded polynomials
\begin{equation}\label{setwitt}
x_1^{(g)}x_2^{(h)},\quad [g,h]\ne1,\,g,h\in G.
\end{equation}
Let $\mathscr{V}_G$ be the variety of $G$-graded associative algebras satisfying the identities \eqref{setwitt}, and $\mathscr{W}_G$ be the variety of $G$-graded associative-Lie pairs satisfying the weak-polynomial identities \eqref{setwitt}. Let $X_G$ be a set of $G$-graded variables, $\mathbb{F}_G(X_G)$ be the relatively free algebra in $\mathscr{V}_G$, freely generated by $X_G$, and $(\mathcal{L},\mathscr{F})$ be a free pair in $\mathscr{W}_G$, freely generated by $X_G$. Then:
\begin{enumerate}
\renewcommand{\labelenumi}{(\roman{enumi})}
\setlength{\itemsep}{.5em}
\item The Lie subalgebra $\mathcal{L}_G(X_G)$ of $\mathbb{F}_G(X_G)$ generated by $X_G$, is the free $G$-graded Lie algebra, freely generated by $X_G$. Furthermore, $\mathrm{SU}_G(\mathcal{L}_G(X_G))=\mathbb{F}_G(X_G)$.
\item The Lie subalgebra of $\mathcal{F}$ generated by $X_G$ coincides with $\mathcal{L}$, and it is the free $G$-graded Lie algebra, freely generated by $X_G$. Moreover, $U_G(\mathcal{L})=\mathcal{F}$.\qed
\end{enumerate}
\end{Cor}

\section{A graded version of Ado's Theorem}
In this section, we assume that $\mathrm{char}\,\mathbb{F}=0$.

If $\mathcal{L}$ is a $G$-graded Lie algebra, then it is known that its center is a graded ideal. Moreover, the following result due to Gordienko (also proved by Pagon, Repov\v s and Zaicev in \cite{PRZ}) is useful:
\begin{Thm}[{\cite[Corollary 4.2 and Corollary 4.3]{Gord16}}]\label{Gord}
Let $\mathcal{L}$ be a finite-dimensional $G$-graded Lie algebra over a field of characteristic zero, and $\mathcal{R}$ its solvable radical.
\begin{enumerate}
\item The nilradical and the solvable radical of $\mathcal{L}$ are graded ideals.
\item (Graded Levi's decomposition) There exists a graded subalgebra $\mathcal{L}_1$ such that $\mathcal{L}_1\cap\mathcal{R}=0$ and $\mathcal{L}=\mathcal{R}+\mathcal{L}_1$.
\end{enumerate}
\end{Thm}

We shall follow the classical ungraded version of the proof of Ado's Theorem. The main idea is to find a graded ideal $I\subseteq U_G(\mathcal{L})$ of finite codimension such that $\mathcal{L}\cap I=0$.

\begin{Lemma}\label{goodideal}
Let $\mathcal{L}$ be a $G$-graded solvable Lie algebra, $\mathcal{R}$ its nilradical, and $I\subseteq U_G(\mathcal{L})$ a graded ideal of finite codimension. Assume that every $x\in\mathcal{R}$ is nilpotent modulo $I$. Then there exists a graded ideal $J\subseteq U_G(\mathcal{L})$ such that:
\begin{enumerate}
\item $J\subseteq I$,
\item $J$ has finite codimension in $U_G(\mathcal{L})$,
\item every $x\in\mathcal{R}$ is nilpotent modulo $J$,
\item every derivation $D$ of $\mathcal{L}$, that can be extended to $U_G(\mathcal{L})$, satisfies $DJ\subseteq J$.
\end{enumerate}
\end{Lemma}
\begin{proof}
Let $I'$ be the graded ideal of $U_G(\mathcal{L})$ generated by $I$ and $\mathcal{R}$. Then $\mathcal{L}/\mathcal{L}\cap I\cong(\mathcal{L}+I)/I$ is a $G$-graded Lie algebra, and $I'/I$ is one of its graded ideals. Since $\dim\mathcal{L}/\mathcal{L}\cap I<\infty$ and $I'/I$ contains a basis constituted of nilpotent elements, then $I'/I$ is nilpotent. Hence there exists $m\in\mathbb{N}$ such that
\[
J:=(I')^m\subseteq I.
\]
Clearly $J$ satisfies (1)--(3). On the other hand, since $\mathrm{char}\,\mathbb{F}=0$, it is known that $D\mathcal{L}\subseteq\mathcal{R}$, for any derivation $D$ of $\mathcal{L}$. Hence $DU_G(\mathcal{L})\subseteq I'$. Thus $DJ\subseteq J$.
\end{proof}

The following lemma is easy to deduce, and we include its statement for the sake of completeness.
\begin{Lemma}\label{innerdev}
Let $\mathcal{L}$ be a finite-dimensional $G$-graded Lie algebra. Then, $\mathrm{IDer}(\mathcal{L})$, the set of inner derivations of $\mathcal{L}$, is a $G$-graded Lie algebra.\qed
\end{Lemma}

\begin{Lemma}
Let $\mathcal{L}$ be a finite-dimensional $G$-graded Lie algebra, and $\mathcal{H}\subseteq\mathrm{End}_\mathbb{F}\mathcal{L}$ be a $G$-graded Lie algebra contained in $\mathrm{Der}(\mathcal{L})$. Then the Lie subalgebra generated by $\mathcal{H}$ and $\mathrm{IDer}(\mathcal{L})$ is a $G$-graded Lie algebra.
\end{Lemma}
\begin{proof}
Since $\mathrm{IDer}(\mathcal{L})$ is a Lie ideal of $\mathrm{Der}(\mathcal{L})$, the vector space $\mathcal{H}+\mathrm{IDer}(\mathcal{L})$ is a Lie algebra. Since both spaces are graded, their sum is a graded subspace as well. As each one of $\mathcal{H}$ and $\mathrm{IDer}(\mathcal{L})$ is a graded Lie subalgebras, it is enough to show the following property. If $D\in\mathcal{H}$ and $x\in\mathcal{L}$ are homogeneous, then $[D,\mathrm{ad}\,x]$ is either $0$ or is homogeneous of degree $\deg_GD\deg_Gx$. This is clear, since $[D,\mathrm{ad}\,x]=\mathrm{ad}(D(x))$, and $D(x)$ is either $0$ or $\deg_GD(x)=\deg_GD\deg_Gx$.
\end{proof}

\begin{Lemma}\label{img}
Let $\mathcal{L}$ and $\mathcal{H}$ be finite-dimensional $G$-graded Lie algebras. Let $\rho\colon \mathcal{H}\to\mathrm{End}_\mathbb{F}\mathcal{L}$ be a Lie homomorphism and a $G$-graded linear map. Then $\mathrm{Im}\,\rho$ is a $G$-graded Lie algebra, and $\rho$ is a graded Lie homomorphism.
\end{Lemma}
\begin{proof}
We know that $\mathrm{Im}\,\rho$ is a graded subspace. So, it is enough to show that, given homogeneous $u$, $v\in\mathrm{Im}\,\rho$, either $[u,v]=0$ or $[u,v]$ is homogeneous of degree $h:=\deg_Gu\deg_Gv$. Let $u_0$, $v_0\in\mathcal{H}$ be homogeneous elements such that $u=\rho(u_0)$ and $v=\rho(v_0)$. Thus, $[u,v]=\rho[u_0,v_0]$, so it is enough to show that this last one is a graded linear map of degree $h$. If $[\deg_gu,\deg_gv]\ne1$, then $[u_0,v_0]=0$, and there is nothing to do. Otherwise, if $x\in\mathcal{L}$ is homogeneous then $uvx$ and $vux$ are both homogeneous elements of degree $h\deg_Gx$. Hence, $[u,v]$ is homogeneous of degree $h$.
\end{proof}

\begin{Example} Let $G=\langle\alpha_1,\alpha_2,\alpha_3\rangle$ be the free group, freely generated by $\{\alpha_1,\alpha_2,\alpha_3\}$. Let $\mathcal{L}$ be the $3$-dimensional abelian Lie algebra, and assume that $\{x_1,x_2,x_3\}$ is a basis of $\mathcal{L}$. Then, imposing $\deg_Gx_i=\alpha_i$, for $i=1$, 2, 3, we obtain a (Lie algebra) $G$-grading on $\mathcal{L}$. Note that $\mathrm{IDer}\,\mathcal{L}=0$ and $\mathrm{Der}\,\mathcal{L}=\mathrm{End}_\mathbb{F}\mathcal{L}$. However, $\mathrm{Der}\,\mathcal{L}$ is \textbf{not} a $G$-graded Lie algebra. Indeed, let $e_{ij}\in\mathrm{End}_\mathbb{F}\mathcal{L}$ be the map defined by
\[
e_{ij}x_\ell=\delta_{j\ell}x_i.
\]
Then $e_{ij}$ is a homogeneous map of degree $\alpha_i\alpha_j^{-1}$. On one hand, $e_{12}e_{23}=e_{13}\ne0$. On the other hand, $[\deg_Ge_{12},\deg_Ge_{23}]\ne1$. Thus, by Lemma \ref{lem1}, it is not possible to have a structure of a $G$-graded Lie algebra on $\mathrm{Der}\,\mathcal{L}$, given by the commutator.

Note that $\mathrm{Der}\,\mathcal{L}$ contains several $G$-graded Lie algebras, with respect to the commutator. For instance, if $1\le i<j\le3$, then $\mathrm{Span}\{e_{ii},e_{ij},e_{ji},e_{jj}\}$ is a $G$-graded Lie algebra.
\end{Example}

The proposition below is the main step in the proof of our main result of this section.
\begin{Prop}\label{mainstep}
Let $\mathcal{L}=\mathcal{R}+\mathcal{L}_1$ be a finite-dimensional $G$-graded Lie algebra, where $\mathcal{R}$ is a graded solvable ideal, and $\mathcal{L}_1$ is a graded subalgebra. Assume that $J\subseteq U_G(\mathcal{R})$ is a graded ideal satisfying (ii)--(iv) of Lemma \ref{goodideal}. Then there exists a graded ideal $J'\subseteq U_G(\mathcal{L})$ such that
\begin{enumerate}
\renewcommand{\labelenumi}{(\roman{enumi})}
\item $\mathcal{R}\cap J'\subseteq\mathcal{R}\cap J$,
\item if $x\in\mathfrak{n}(\mathcal{R})$ and $y\in\mathcal{L}_1$ is such that $\mathrm{ad}_\mathcal{R}\,y$ is nilpotent, then $x+y$ is nilpotent modulo $J'$.
\end{enumerate}
\end{Prop}
\begin{proof}
Consider the map
\[
\mathrm{IDer}\,\mathcal{L}\to\mathrm{Der}(\mathcal{R}/\mathcal{R}\cap J),
\]
given by the restriction of the inner derivation to $\mathcal{R}$, followed by the induced map on the quotient. Since $\mathcal{R}\cap J$ is an ideal of $\mathcal{R}$, the map is well-defined. Since $\mathcal{R}$ is a graded ideal of $\mathcal{L}$, it follows that the map is graded. Thus, by Lemma \ref{img}, its image is a $G$-graded Lie algebra, say $\mathcal{H}$. Now, consider the composition
\[
\rho\colon \mathcal{L}\to\mathrm{IDer}\,\mathcal{L}\to\mathcal{H}.
\]
Therefore $\rho$ is a $G$-graded homomorphism of Lie algebras, and it admits an extension $\bar{\rho}\colon U_G(\mathcal{L})\to U_G(\mathcal{H})$. Let $J'=\mathrm{Ker}\,\bar{\rho}\cap\langle J\rangle$, where $\langle J\rangle=U_G(\mathcal{L})JU_G(\mathcal{L})$. Note that, by construction, $\mathcal{R}\cap J'\subseteq\mathcal{R}\cap J$. Moreover, if $y\in\mathcal{L}_1$ is such that $\mathrm{ad}_\mathcal{R}\,y$ is nilpotent, then by construction $\bar{\rho}(y)$ is nilpotent too. Thus, $y$ is nilpotent modulo $J'$.
\end{proof}

Now we proceed with the proof of our main result. We divide the proof in several steps.
\begin{Lemma}\label{rep_center}
Let $\mathfrak{z}$ be a finite-dimensional $G$-graded abelian Lie algebra. Then there exists a graded ideal of finite codimension $I\subseteq U_G(\mathfrak{z})$ such that
\begin{enumerate}
\renewcommand{\labelenumi}{(\roman{enumi})}
\item $I\cap\mathfrak{z}=0$, and
\item every $x\in\mathfrak{z}$ is nilpotent modulo $I$.
\end{enumerate}
\end{Lemma}
\begin{proof}
Since $\mathfrak{z}$ has zero product, it is also an associative algebra with trivial product. So $(\mathfrak{z},\mathfrak{z})$ is a pair. Hence, the identity map $\mathfrak{z}\to\mathfrak{z}$ admits an extension $\rho\colon U_G(\mathfrak{z})\to\mathfrak{z}$. Let $I=\mathrm{Ker}\,\rho$. Since $\dim\mathfrak{z}<\infty$, then $I$ has finite codimension. Since $\rho$ is a graded homomorphism, $I$ is also a graded ideal. As $\rho$ is 1--1 in $\mathfrak{z}$, we see that $I\cap\mathfrak{z}=0$. Finally, once $\mathfrak{z}$ has the trivial product,
$$
\rho(U_G(\mathfrak{z})^2)\subseteq\mathfrak{z}^2=0.
$$
In particular, $\mathfrak{z}^2\subseteq I$.
\end{proof}

\begin{Lemma}\label{rep_nilrad}
Let $\mathfrak{n}$ be a finite-dimensional $G$-graded nilpotent Lie algebra, and let $\mathfrak{z}$ be its center. Then there exists a graded ideal of finite codimension $I\subseteq U_G(\mathfrak{z})$ such that
\begin{enumerate}
\renewcommand{\labelenumi}{(\roman{enumi})}
\item $I\cap\mathfrak{z}=0$, and
\item every $x\in\mathfrak{n}$ is nilpotent modulo $I$.
\end{enumerate}
\end{Lemma}
\begin{proof}
Consider a chain
$$
\mathfrak{z}=\mathfrak{n}_0\subseteq\mathfrak{n}_1\subseteq\cdots\subseteq\mathfrak{n}_t=\mathfrak{n},
$$
where each $\mathfrak{n}_i$ is a graded ideal and $\dim\mathfrak{n}_i=i+\dim\mathfrak{z}$. We shall prove the result by induction on $i$. If $i=0$, then we may apply Lemma \ref{rep_center}. So, for some $0\le i<i$, assume that $I_i\subseteq U_G(\mathfrak{n}_i)$ is a graded ideal satisfying (i), and such that each $x\in\mathfrak{n}_i$ is nilpotent modulo $I_i$. Write $\mathfrak{n}_{i+1}=\mathfrak{n}_i+\mathcal{L}_i$, where $\mathcal{L}_i$ is a graded Lie algebra. By Proposition \ref{mainstep}, there exists a graded ideal $I_{i+1}\subseteq U_G(\mathfrak{n}_{i+1})$ satisfying (ii) and such that $I_{i+1}\cap\mathfrak{z}_i\subseteq I_i\cap\mathfrak{z}_i$. Thus,
$$
I_{i+1}\cap\mathfrak{z}=I_{i+1}\cap\mathfrak{z}\cap\mathfrak{n}_i\subseteq I_i\cap\mathfrak{n}_i\cap\mathfrak{z}=0.
$$
Hence, $I_{i+1}$ satisfies (i) as well.
\end{proof}

\begin{Lemma}\label{rep_rad}
Let $\mathcal{R}$ be a finite-dimensional $G$-graded solvable Lie algebra, $\mathfrak{n}$ its nilradical, and $\mathfrak{z}$ its center. Then there exists a graded ideal of finite codimension $I\subseteq U_G(\mathcal{R})$ such that
\begin{enumerate}
\renewcommand{\labelenumi}{(\roman{enumi})}
\item $I\cap\mathfrak{z}=0$, and
\item every $x\in\mathfrak{n}$ is nilpotent modulo $I$.
\end{enumerate}
\end{Lemma}
\begin{proof}
Consider a chain of graded subalgebras
$$
\mathfrak{n}=\mathcal{R}_0\subseteq\mathcal{R}_1\subseteq\cdots\subseteq\mathcal{R}_s=\mathcal{R},
$$
where $\mathcal{R}_i$ is an ideal of $\mathcal{R}_{i+1}$, and $\dim\mathcal{R}_i=i+\dim\mathfrak{n}$. Then, it is known that the nilradical of each $\mathcal{R}_i$ is $\mathfrak{n}$. We may proceed by induction on $i$, where the validity for $i=0$ holds thanks to Lemma \ref{rep_nilrad}. Then, we may repeat the argument given in the proof of the previous lemma.
\end{proof}

\begin{Cor}\label{rep_alg}
Let $\mathcal{L}$ be a finite-dimensional $G$-graded Lie algebra and $\mathfrak{z}$ its center. Then there exists a graded ideal of finite codimension $I\subseteq U_G(\mathcal{L})$ such that $I\cap\mathfrak{z}=0$.
\end{Cor}
\begin{proof}
It is enough to combine the graded Levi's decomposition (Theorem \ref{Gord}), Lemma \ref{rep_rad} and Proposition \ref{mainstep}.
\end{proof}

Now, we are in a position to prove the main result of this section.
\begin{Thm}\label{graded_ado}
Let $\mathcal{L}$ be a finite-dimensional $G$-graded Lie algebra over a field of characteristic zero, where $G$ is a non-necessarily abelian group. Then there exists a finite-dimensional $G$-graded associative algebra $\mathcal{A}$ such that $(\mathcal{L},\mathcal{A})$ is a $G$-pair. In other words $\mathcal{L}\subseteq\mathcal{A}$ is a graded vector subspace, and $\mathcal{L}$ is a $G$-graded Lie algebra with respect to the commutator of $\mathcal{A}$.
\end{Thm}
\begin{proof}
It is enough to find a graded ideal $I\subseteq U_G(\mathcal{L})$ of finite codimension such that $I\cap\mathcal{L}=0$. If $I$ is such an ideal then $\mathcal{A}=U_G(\mathcal{L})/I$ is a finite-dimensional $G$-graded associative algebra, and the natural graded linear map $\mathcal{L}\to\mathcal{A}$ has $\mathcal{L}\cap I=0$ as its kernel. Thu, $(\mathcal{L},\mathcal{A})$ is a pair.

Consider the adjoint map $\mathcal{L}\to\mathrm{IDer}\,\mathcal{L}$, and its extension $U_G(\mathcal{L})\to\mathrm{End}_\mathbb{F}(\mathcal{L})$. Let $I_1$ be the kernel of the latter map. Then it is known that $I_1\cap\mathcal{L}=\mathfrak{z}$, the center of $\mathcal{L}$. Since $\dim\mathrm{End}_\mathbb{F}\,\mathcal{L}<\infty$, we see that $I_1$ has finite codimension. Let $I_2$ be the ideal of Corollary \ref{rep_alg}, and let $I=I_1\cap I_2$. Then, by construction, $I$ is a graded ideal of finite codimension. Moreover,
\[
I\cap\mathcal{L}=(I_1\cap\mathcal{L})\cap I_2=\mathfrak{z}\cap I_2=0.
\]
The proof is complete.
\end{proof}

\begin{Remark}
There is no guarantee that we have a strong pair $(\mathcal{L},\mathcal{A})$ where $\mathcal{A}$ is finite-dimensional. This is because we cannot guarantee that $(\mathrm{IDer}\,\mathcal{L},\mathrm{End}_\mathbb{F}(\mathcal{L}))$ is a strong pair.
\end{Remark}

\section{Group gradings on Lie algebras\label{Lie_gr_is_ab}}

The following question was posed by I.~Shestakov: is any $G$-grading on a Lie algebra equivalent to an abelian group grading? We shall use our constructions to provide a partial answer to this question. To this end we need several technical lemmas.

First we need a map that behaves like a homomorphism from a free abelian group to an arbitrary group. This construction is useful for finding an equivalence between an abelian grading and a $G$-grading, where $G$ is not necessarily abelian.
\begin{Lemma}\label{mainlemma}
Let $G$ be a group, $g_1$, \dots, $g_m\in G$, and $Z=\langle a_1,\ldots,a_m\mid [a_i,a_j]=1,\forall i,j\rangle$ be the free abelian group of rank $m$. Let
\[
S=\{a_{i_1}\cdots a_{i_r}\in Z\mid\{g_{i_1},\ldots,g_{i_r}\}\text{ \gab}\}.
\]
Then there exists a well-defined map $\alpha:S\to G$ such that
\[
\alpha(a_{i_1}\cdots a_{i_r})=g_{i_1}\cdots g_{i_r}.
\]
\end{Lemma}
\begin{proof}
Let
\[
\mathscr{F}=\left\{\{i_1,\ldots,i_r\}\subseteq\{1,\ldots,m\}\mid\{g_{i_1},\ldots,g_{i_r}\}\text{ \gab}\right\}.
\]
If $S_1$, $S_2\in\mathscr{F}$, then $S_1\cap S_2\in\mathscr{F}$. For each $S_0=\{i_1,\ldots,i_r\}\in\mathscr{F}$, let $G_{S_0}=\langle g_{i_1},\ldots g_{i_r}\rangle\subseteq G$ and $Z_{S_0}=\langle a_{i_1},\ldots, a_{i_r}\rangle$. Then $Z_{S_0}$ is a free abelian group of rank $r$, and $G_{S_0}$ is an abelian subgroup of $G$. Thus there is a well-defined group homomorphism $\alpha_{S_0}\colon Z_{S_0}\to G_{S_0}$ such that $a_{i_j}\mapsto g_{i_j}$, for each $i_j\in S_0$. Note that $S\subseteq\bigcup_{S_0\in\mathscr{F}}Z_{S_0}$. Now for each $s\in S$, let $s\in S_0\in\mathscr{F}$, and set
\[
\alpha(s)=\alpha_{S_0}(s).
\]
If $s\in S_1\cap S_2$, then since the maps are uniquely defined by their values on the generators, we have
\[
\alpha_{S_1}(s)=\alpha_{S_1\cap S_2}(s)=\alpha_{S_2}(s).
\]
Therefore $\alpha$ is well-defined.
\end{proof}

Let us consider once again the variety $\mathscr{V}_G$, defined in \Cref{variety}, that is, the variety of $G$-graded associative algebras satisfying the identities
\[
x_1^{(g)}x_2^{(h)}=0,\quad[g,h]\ne1,\,g,h\in G.
\]
Let $\mathcal{L}$ be a $G$-graded Lie algebra and $\mathcal{B}=\{e_i\mid i\in N\}$ a homogeneous vector space basis. We let $X_\mathcal{B}=\{x_i^{(g_i)}\mid i\in N,\,g_i=\deg_Ge_i\}$. Assume that $\{g_i\mid i\in N\}$ has $m$ elements, and denote these by $\{g_1,\ldots,g_m\}$. Let $Z=\langle a_1,\ldots,a_m\mid [a_i,a_j]=1,\forall i,j\rangle$ be the free abelian group of rank $m$. We consider the variables $X_Z=\{x_i^{(z_j)}\mid \deg e_i=g_j,\,i\in N\}$, the ``lifting" of the variables $X_\mathcal{B}$ to $Z$-graded variables. Let $\mathbb{F}\langle X_Z\rangle$ be the free associative $Z$-graded algebra, freely generated by $X_Z$, and $\mathbb{F}_G(X_\mathcal{B})$ the relatively free algebra in $\mathscr{V}_G$, freely generated by $X_\mathcal{B}$. We have an algebra epimorphism $\psi\colon \mathbb{F}\langle X_Z\rangle\to\mathbb{F}_G(X_\mathcal{B})$, where $\psi(x_i^{(z_j)})=x_i^{(g_j)}$, for each $i\in N$. Following \Cref{mainlemma}, let
\[
S=\{z_{i_1}\cdots z_{i_r}\in Z\mid\{g_{i_1},\ldots,g_{i_r}\}\text{ \gab}\},
\]
and let $\alpha\colon S\to G$ be defined by
\[
\alpha(z_{i_1}\cdots z_{i_r})=g_{i_1}\cdots g_{i_r}.
\]

\begin{Prop}\label{psi_is_alpha}
The map $\psi$ is $\alpha$-graded, that is,
\[
\psi\left(\mathbb{F}\langle X_Z\rangle\right)_z\subseteq\left(\mathbb{F}_G(X_\mathcal{L})\right)_{\alpha(z)},\quad\forall z\in S,
\]
and $\psi\left(\mathbb{F}\langle X_Z\rangle\right)_z=0$, for $z\in Z\setminus S$.
\end{Prop}
\begin{proof}
Let $m=x_{i_1}^{(z_{i_1})}\cdots x_{i_r}^{(z_{i_r})}\in\mathbb{F}\langle X_Z\rangle$. It is enough to show that either $\psi(m)=0$ or $\psi(m)$ is homogeneous  of degree $\alpha(z_{i_1})\cdots\alpha(z_{i_r})$ in the $G$-grading. The proof is by induction on $r$, with obvious basis the case $r=1$. Let $m_0=x_{i_1}^{(z_{i_1})}\cdots x_{i_{r-1}}^{(z_{i_{r-1}})}$. There is nothing to do if $\psi(m_0)=0$. Otherwise, the subset $\{\alpha(z_{i_1}),\ldots,\alpha(z_{i_{r-1}})\}$ generates an abelian group. If, for some $j$,
\[
[\alpha(z_{i_j})\alpha(z_{i_{j+1}})\cdots\alpha(z_{i_{r-1}}),\alpha(z_{i_r})]\ne1,
\]
then $\psi(m)=0$. Otherwise, $\{\alpha(z_{i_1}),\ldots,\alpha(z_{i_{r-1}}),\alpha(z_{i_r})\}$ generates an abelian subgroup as well. Hence, either $\psi(m)=0$ or $\deg_G\psi(m)=\alpha(\deg_Z m)$.
\end{proof}

\begin{Cor}\label{F_G_is_coarsening_ab}
The grading of $\mathbb{F}_G(X_\mathcal{B})$ is a coarsening of a grading realized by the respective universal abelian grading group.
\end{Cor}
\begin{proof}
It follows from \Cref{psi_is_alpha} that there is a fine abelian grading on $\mathbb{F}_F(X_Z)$ (obtained as a quotient of $\mathbb{F}\langle X_Z\rangle$) that is a refinement of $\mathbb{F}_G(X_\mathcal{B})$.
\end{proof}

We do not know yet a proof for the following.
\begin{Conj}\label{coarsening_ab_is_ab}
Let $\Gamma$ be a group grading on a Lie algebra given by a coarsening of an abelian grading. Then $\Gamma$ is abelian.
\end{Conj}

\begin{Example}
\Cref{coarsening_ab_is_ab} is false in a context other than of a Lie algebra. Indeed, we know that on any matrix algebra $M_n(\mathbb{F})$ there is a unique fine elementary grading, up to an equivalence, and it may be realized by an abelian group (see, for instance, \cite[Proposition 2.31]{EK13}). Thus, any elementary grading on a matrix algebra is a coarsening of an abelian grading. On the other hand, it is not hard to see that an abelian grading on a finite-dimensional algebra is equivalent to a finite group grading (see also \cite[Proposition 4.11]{GordSch}). In \cite[Corollary 5.6]{GordSch}, it is proved that there exists an elementary grading $\Gamma$ on a matrix algebra $M_n(\mathbb{F})$ that cannot be realized by a finite group. As a consequence, $\Gamma$ cannot be realized by an abelian group, and it is a coarsening of an abelian grading.
\end{Example}

Putting the pieces together (\Cref{F_G_is_coarsening_ab} and \Cref{coarsening_ab_is_ab}), we obtain the following lemma.
\begin{Lemma}\label{SU_G_abelian}
Let $\mathcal{L}$ be a $G$-graded Lie algebra. Then, if \Cref{coarsening_ab_is_ab} is true, the grading on $\mathrm{SU}_G(\mathcal{L})$ is equivalent to an abelian grading.
\end{Lemma}
\begin{proof}
From \Cref{F_G_is_coarsening_ab} and \Cref{coarsening_ab_is_ab}, we obtain that the grading on $\mathbb{F}_G(X_\mathcal{B})$ is equivalent to an abelian grading. Now the equivalence holds for the ideal $J_\mathcal{B}$ and its quotient $\mathbb{F}_G(X_\mathcal{B})/J_\mathcal{B}\cong\mathrm{SU}_G(\mathcal{L})$.
\end{proof}

As a consequence, we obtain the main result of this section.
\begin{Thm}\label{Liegr_isab}
If \Cref{coarsening_ab_is_ab} is true, then every $G$-grading on a Lie algebra is equivalent to an abelian grading.
\end{Thm}
\begin{proof}
Let $\mathcal{L}$ be a $G$-graded algebra. By construction (\Cref{strong_gr_universal}), $\mathcal{L}$ is a graded subalgebra of $\mathrm{SU}_G(\mathcal{L})$. From \Cref{SU_G_abelian}, the grading on $\mathrm{SU}_G(\mathcal{L})$ is equivalent to an abelian grading. Hence the grading on $\mathcal{L}$ is equivalent to an abelian grading as well.
\end{proof}

We summarize the results of this section. We obtained the equivalence of the following two problems.
\begin{Cor}\label{problem_equiv}
The following assertions are equivalent.
\begin{enumerate}
\item Every group grading on a Lie algebra is equivalent to an abelian group grading.
\item Every coarsening of an abelian group grading on a Lie algebra is equivalent to an abelian group grading.
\end{enumerate}
\end{Cor}
\begin{proof}
Clearly (1) implies (2). Now, assuming that (2) holds valid, \Cref{Liegr_isab} proves (1).
\end{proof}

\section*{Acknowledgements}
We are thankful to Professor Ivan Shestakov, Plamen Koshlukov, Lucia Ikemoto and Mikhail Kochetov for the useful discussion and encouragement.

\end{document}